\documentclass{article}
\usepackage{xpatch,amsthm}

\usepackage[titletoc]{appendix}
\usepackage{graphicx}
\graphicspath{ {images/} }
\usepackage{amssymb}
\usepackage{mathtools}
\usepackage{hyperref}
\usepackage[capitalise]{cleveref}
\usepackage{amsmath}
\usepackage{amsthm}
\usepackage{bbm}
\usepackage[numbers]{natbib}
\newtheorem{theorem}{Theorem}
\numberwithin{theorem}{section}
\newtheorem{lemma}[theorem]{Lemma}
\crefname{lemma}{Lemma}{lemma}

\usepackage{color}

\usepackage{float}
\makeatletter
   \xpatchcmd{\@thm}{\fontseries\mddefault\upshape}{}{}{} 
\makeatother
\newtheorem{example}[theorem]{Example}
\newtheorem*{q*}{Question}

\newtheorem{assumption}[theorem]{Assumption}
\crefname{assumption}{Assumption}{assumption}

\numberwithin{equation}{section}

\DeclareMathOperator*{\esssup}{ess\,sup}
\def\defto{\buildrel def\over =}
\def\P{\mathbb{P}}
\def\S{\mathcal{S}}
\theoremstyle{remark}
\newtheorem*{remark}{Remark}
\makeatletter
\let\@fnsymbol\@arabic
\makeatother

\author{Saul D. Jacka\thanks{Saul D. Jacka gratefully acknowledges funding received from the EPSRC grant EP/P00377X/1 and is also grateful to the Alan Turing Institute for their financial support under the EPSRC 
grant EP/N510129/1. E-mail: \textit{s.d.jacka@warwick.ac.uk}}\hspace{3mm}  and  \hspace{3mm}Dominykas Norgilas\thanks{Dominykas Norgilas gratefully acknowledges funding received from the EPSRC Doctoral Training Partnerships grant EP/M508184/1. E-mail: \textit{d.norgilas@warwick.ac.uk}}\\
Department of Statistics, University of Warwick\\Coventry CV4 7AL, UK}

\title{On the compensator in the Doob-Meyer decomposition of the Snell envelope}

\begin{document}

\maketitle

\begin{abstract}
Let $G$ be a semimartingale, and $S$ its Snell envelope. Under the assumption that $G\in\mathcal{H}^1$, we show that the finite-variation part of $S$ is absolutely continuous with respect to the decreasing part of the finite-variation part of $G$. In the Markovian setting, this enables us to identify sufficient conditions for the value function of the optimal stopping problem to belong to the domain of the extended (martingale) generator of the underlying Markov process. We then show that the \textit{dual} of the optimal stopping problem is a stochastic control problem for a controlled Markov process, and the optimal control is characterised by a function belonging to the domain of the martingale generator. Finally, we give an application to the smooth pasting condition.
\end{abstract}

\section{Introduction}\label{intro}
Given a (gains) process $G=(G_t)_{t\geq0}$, living on the usual filtered probability space $(\Omega, \mathcal{F}, \mathbb{F}=(\mathcal{F}_t)_{t\geq0},\mathbb{P})$, the classical optimal stopping problem is to find a maximal reward $v(0)=\sup_{\tau\geq0}\mathbb{E}[G_\tau]$, where the supremum is taken over all $\mathbb{F}$ - stopping times. In order to compute $v(0)$, we consider, for each $\mathbb{F}$ - stopping time $\sigma\geq0$, the value function $v(\sigma)=\esssup_{\tau\geq\sigma}\mathbb{E}[G_\tau\lvert\mathcal{F}_\sigma]$. It is, or should be, well-known (see, for example, El Karoui \cite{el1981lesaspects}, Karatzas and Shreve \cite{KS1998methods}) that under suitable integrability and regularity conditions on the process $G$,  the Snell envelope of $G$, denoted by $S=(S_t)_{t\geq0}$, is the minimal supermartingale which dominates $G$ and aggregates the value function $v$, so that for any $\mathbb{F}$ - stopping time $\sigma\geq0$, $S_\sigma=v(\sigma)$ almost surely. Moreover, $\tau_\sigma:=\inf\{r\geq \sigma:S_r=G_r\}$ is the minimal optimal stopping time, so, in particular, $S_{\sigma}=v(\sigma)=\mathbb{E}[G_{\tau_\sigma}\lvert\mathcal{F}_\sigma]$ almost surely. A successful construction of the process $S$ leads, therefore, to the solution of the initial optimal stopping problem.

In the Markovian setting the gains process takes the form $G=g(X)$, where $g(\cdot)$ is some payoff function applied to an underlying Markov process $X$. Under very general conditions, the Snell envelope is then characterised as the least super-mean-valued function $V(\cdot)$ that majorizes $g(\cdot)$. A standard technique to find the value function $V(\cdot)$ is to solve the corresponding obstacle (free-boundary) problem. For an exposition of the general theory of optimal stopping in both settings we also refer to Peskir and Shiryaev \cite{peskir2006optimal}.

The main aim of this paper is to answer the following canonical question of interest:
\begin{q*}{When does the value function $V(\cdot)$ belong to the domain of the extended (martingale) generator of the underlying Markov process $X$?}
\end{q*}

Very surprisingly, given how long general optimal stopping problems have been studied (see Snell \cite{snell1952applications}), we have been unable to find any general results about this.

As the title suggests, we tackle the question by considering the optimal stopping problem in a more general (semimartingale) setting first. If a gains process $G$ is sufficiently integrable, then $S$ is of class (D) and thus uniquely decomposes into the difference of a uniformly integrable martingale, say $M$, and a predictable, increasing process, say $A$, of integrable variation. From the general theory of optimal stopping it can be shown that $\bar{\tau}_\sigma:=\inf\{r\geq\sigma:A_r>0\}$ is the maximal optimal stopping time, while the stopped process $S^{\bar{\tau}_\sigma}=(S_{t\wedge\bar{\tau}_\sigma})_{t\geq0}$ is a martingale. Therefore, the finite variation part of $S$, $A$, must be zero up to $\bar{\tau}_\sigma$. Now suppose that $G$ is a semimartingale itself. Then its finite variation part can be further decomposed into the sum of increasing and decreasing processes that are, as random measures, mutually singular. Off the support of the decreasing one, $G$ is (locally) a submartingale, and thus in this case it is suboptimal to stop, and we again expect $S$ to be (locally) a martingale. This also suggests that $A$ increases only if the decreasing component of the finite variation part of $G$ decreases. In particular, we prove the following fundamental result (see \cref{tabscont}):
\begin{equation*}
\parbox{0.8\linewidth}{\textit{the finite-variation process in the Doob-Meyer decomposition of} $S$ \textit{is absolutely continuous with respect to the decreasing part of the corresponding finite-variation process in the decomposition of} $G$.}
\end{equation*}
This being a very natural conjecture, it is not surprising that some variants of it have already been considered. As a helpful referee pointed out to us, several versions of \cref{tabscont} were established in the literature on reflected BSDEs under various assumptions on the gains process, see El Karoui et al. \cite{el1997bsde} ($G$ is a continuous semimartingale), Crep\'{e}y and Matoussi \cite{crepey2008reflected} ($G$ is a c\`{a}dl\`{a}g quasi-martingale), Hamad\'{e}ne and Ouknine \cite{hamadene2003rbsde} ($G$ is a limiting process of a sequence of sufficiently regular semimartingales). We note that these results (except Hamad\'{e}ne and Ouknine \cite{hamadene2003rbsde}, where the assumed regularity of $G$ is exploited) are proved essentially by using (or appropriately extending) the related (but different) result established in Jacka \cite{jacka1993local}. There, under the assumption that $S$ and $G$ are both continuous and sufficiently integrable semimartingales, the author shows that a local time of $S-G$ at zero is absolutely continuous with respect to the decreasing part of the finite-variation process in the decomposition of $G$. Our proof of \cref{tabscont} relies on the classical methods establishing the Doob-Meyer decomposition of a supermartingale.

The first part of \cref{main} is devoted to the groundwork necessary to establish \cref{tabscont}. It turns out that an answer to the motivating question of this paper then follows naturally. In particular, in the second part of \cref{main}, in \cref{pvdom}, we show that, under very general assumptions on the underlying Markov process $X$, if the payoff function $g(\cdot)$ belongs to the domain of the martingale generator of $X$, so does the value function $V(\cdot)$ of the optimal stopping problem.   

In \cref{app} we discuss some applications. First, we consider a dual approach to optimal stopping problems due to Davis and Karatzas \cite{davis1994deterministic} (see also Rogers \cite{rogers2002monte}, and Haugh and Kogan \cite{haugh2004pricing}). In particular, from the absolute continuity result announced above, it follows that the dual is a stochastic control problem \textit{for a controlled Markov process}, which opens the doors to the application of all the available theory related to such problems (see Fleming and Soner \cite{fleming2006controlled}). Secondly, if the value function of the optimal stoping problem belongs to the domain of the martingale generator, under a few additional (but general) assumptions, we also show that the celebrated \textit{smooth fit} principle holds for (killed) one-dimensional diffusions.
\section{Preliminaries}\label{prelim}
\subsection{General framework}\label{prelimmart}
Fix a time horizon $T\in (0,\infty]$. Let $G$ be an adapted, c\`{a}dl\`{a}g gains process on $(\Omega, \mathcal{F},  \mathbb{F}=(\mathcal{F}_t)_{0\leq t\leq T},\mathbb{P})$, where $\mathbb{F}$ is a right-continuous and complete filtration (augmented by the null sets of $\mathcal{F}=\mathcal{F}_T$). We suppose that $\mathcal{F}_0$ is trivial. In the case $T=\infty$, we interpret $\mathcal{F}_\infty=\sigma\Big(\cup_{0\leq t<\infty}\mathcal{F}_t\Big)$ and $G_\infty=\liminf_{t\to\infty}G_t$. For two $\mathbb{F}$-stopping times $\sigma_1$, $\sigma_1$ with $\sigma_1\leq\sigma_2$ $\mathbb{P}$-a.s., by $\mathcal{T}_{\sigma_1,\sigma_2}$ we denote the set of all $\mathbb{F}$-stopping times $\tau$ such that $\mathbb{P}(\sigma_1\leq\tau\leq\sigma_2)=1$. We will assume that the following condition is satisfied:
\begin{equation}\label{sup}
\mathbb{E}\Big[\sup_{0\leq t\leq T}\lvert G_t\lvert\Big]<\infty,
\end{equation}
and let
$$\bar{\mathbb{G}} \text{ be the space of all adapted, c\`{a}dl\`{a}g processes such that \eqref{sup} holds.}
$$

The \textit{optimal stopping problem} is to compute the maximal expected reward
\begin{equation}
v(0):=\sup_{\tau\in\mathcal{T}_{0,T}}\mathbb{E}[G_\tau].
\end{equation}
\begin{remark}
First note that by \eqref{sup}, $\mathbb{E}[G_\tau]<\infty$ for all $\tau\in\mathcal{T}_{0,T}$, and thus $v(0)$ is finite. Moreover, most of the general results regarding optimal stopping problems are proved under the assumption that $G$ is a non-negative (hence the \textit{gains}) process. However, under \eqref{sup}, $N=(N_t)_{0\leq t \leq T}$ given by $N_t=\mathbb{E}[\sup_{0\leq s\leq T}\lvert G_s\lvert\lvert\mathcal{F}_t]$ is a uniformly integrable martingale, while $\hat{G}:=N+G$ defines a non-negative process (even if $G$ is allowed to take negative values). Then
\begin{equation*}
\hat{v}(0):= \sup_{\tau\in\mathcal{T}_{0,T}}\mathbb{E}[N_\tau + G_\tau] = \mathbb{E}\Big[\sup_{0\leq t\leq T}\lvert G_t\lvert\Big] + \sup_{\tau\in\mathcal{T}_{0,T}}\mathbb{E}[ G_\tau],
\end{equation*}
and finding $\hat{v}(0)$ is the same as finding $v(0)$. Hence we may, and shall, assume without loss of generality that $G\geq 0$.
\end{remark}

The key to our study is provided by the family $\{v(\sigma)\}_{\sigma\in\mathcal{T}_{0,T}}$ of random variables
\begin{equation}\label{valuefamily}
v(\sigma):=\esssup_{\tau\in\mathcal{T}_{\sigma,T}}\mathbb{E}[G_\tau\lvert\mathcal{F}_\sigma],\quad\sigma\in\mathcal{T}_{0,T}.
\end{equation}
Note that, since each deterministic time $t\in[0,T]$ is also a stopping time, \eqref{valuefamily} defines an adapted \textit{value} process $(v_t)_{0\leq t\leq T}$ with $v_t=v(t)$. We begin with a fundamental result characterising the so-called Snell envelope process, $S=(S_t)_{0\leq t\leq T}$, of $G$. In particular, $S$ is a version of  $(v_t)_{0\leq t\leq T}$ that aggregates the value function $v(\cdot)$ at each stopping time $\sigma\in\mathcal{T}_{0,T}$ (see Appendix D in Karatzas and Shreve \cite{KS1998methods}). 

\begin{theorem}[Characterisation of $S$]\label{tsnell}
Let $G\in\bar{\mathbb{G}}$. The Snell envelope process $S$ of $G$ satisfies
$S_\sigma=v(\sigma)$ $\mathbb{P}$-a.s., $\sigma\in\mathcal{T}_{0,T}$, and is the minimal c\`{a}dl\`{a}g supermartingale that dominates $G$.
\end{theorem}

For the proof of Theorem \ref{tsnell} under slightly more general assumptions on the gains process $G$ consult Appendix I in Dellacherie and Meyer \cite{delmey} or Proposition 2.26 in El Karoui \cite{el1981lesaspects}. 

 If $G\in\bar{\mathbb{G}}$, it is clear that $G$ is a uniformly integrable process. In particular, it is also of class (D), i.e. the family of random variables $\{G_\tau\mathbbm{1}_{\{\tau<\infty\}}:\tau \textrm{ is a stopping time}\}$ is uniformly integrable. On the other hand, a right-continuous adapted process $Z$ belongs to the class (D) if there exists a uniformly integrable martingale $\hat{N}$, such that, for all $t\in[0,T]$, $\lvert Z_t\lvert\leq \hat{N}_t$ $\mathbb{P}$-a.s. (see e.g. Dellacherie and Meyer \cite{delmey}, Appendix I and references therein). In our case, by the definition of $S$ and using the conditional version of Jensen's inequality, for $t\in[0,T]$, we have
\begin{equation*}
\lvert S_t\lvert\leq\mathbb{E}\Big[\sup_{0\leq s\leq T}\lvert G_s\lvert\Big\lvert\mathcal{F}_t\Big]:=N_t
\quad\mathbb{P}\textrm{-a.s.}
\end{equation*}
But, since $G\in\bar{\mathbb{G}}$, ${N}$ is a uniformly integrable martingale, which proves the following\begin{lemma}\label{lcondition}
Suppose $G\in\bar{\mathbb{G}}$. Then $S$ is of class (D).
\end{lemma}

Let $\mathcal{M}_{0}$ denote the set of right-continuous martingales started at zero. Let $\mathcal{M}_{0,loc}$ and $\mathcal{M}_{0,UI}$ denote the spaces of local and uniformly integrable martingales (started at zero), respectively. Similarly, the adapted processes of finite and integrable variation will be denoted by $FV$ and $IV$, respectively. 

It is well-known that a right-continuous (local) supermartingale $P$ has a unique decomposition $P=B-I$ where $B\in\mathcal{M}_{0,loc}$ and $I$ is an increasing ($FV$) process which is predictable. This can be regarded as the general Doob-Meyer decomposition of a supermartingale. Specialising to class (D) supermartingales we have a stronger result (this is a consequence of, for example, Protter \cite{protter2005stochastic} Theorem 16, p.116 and Theorem 11, p.112): 
\begin{theorem}[Doob-Meyer decomposition]\label{DB}
Let $G\in\bar{\mathbb{G}}$. Then the Snell envelope process $S$ admits a unique decomposition
\begin{equation}\label{DM}
S=M^*-A,
\end{equation}
where $M^*\in\mathcal{M}_{0,UI}$, and $A$ is a predictable, increasing $IV$ process.
\end{theorem}
\begin{remark}
It is normal to assume that the process $A$ in the Doob-Meyer decomposition of $S$ is started at zero. The duality result alluded to in the introduction is one reason why we do not do so here.
\end{remark}

An immediate consequence of \cref{DB} is that $S$ is a semimartingale. In addition, we also assume that $G$ is a semimartingale with the following decomposition: 
\begin{equation}\label{dec}
G=N+D,
\end{equation}
where $N\in\mathcal{M}_{0,loc}$ and $D$ is a $FV$ process. Unfortunately, the decomposition \eqref{dec}  is not, in general, unique. On the other hand, uniqueness {\em is} obtained by requiring the $FV$ term to also be predictable, at the cost of restricting only to locally integrable processes. If there exists a decomposition of a semimartingale $X$ with a predictable $FV$ process, then we say that $X$ is $special$. For a special semimartingale we always choose to work with its $canonical$ decomposition (so that a $FV$ process is predictable). Let $$\mathbb{G} \text{ be the space of semimartingales in }\bar{\mathbb{G}}.
$$ 
\begin{lemma}\label{lunique}
Suppose $G\in\mathbb{G}$. Then $G$ is a special semimartingale.
\end{lemma}
See Theorems 36 and 37 (p.132) in Protter \cite{protter2005stochastic} for the proof.

The following lemma provides a further decomposition of a semimartingale (see Proposition 3.3 (p.27) in Jacod and Shiryaev \cite{jacod2013limit}). In particular, the $FV$ term of a special semimartingale can be uniquely (up to initial values) decomposed in a  predictable way, into the difference of two increasing, mutually singular $FV$ processes. 
\begin{lemma}\label{ldpredict}
Suppose that $K$ is a c\`{a}dl\`{a}g, adapted process such that $K\in FV$. Then there exists a unique pair $(K^+,K^-)$ of adapted increasing processes such that $K-K_0=K^+-K^-$ and $\int\lvert dK_s\lvert=K^++K^-$. Moreover, if $K$ is predictable, then $K^+$, $K^-$ and $\int\lvert dK_s\lvert$ are also predictable.
\end{lemma}

\subsection{Markovian setting}\label{prelimmark}
\paragraph{The Markov process}
Let $(E,\mathcal{E})$ be a metrizable Lusin space endowed with the $\sigma$-field of Borel subsets of $E$. Let $X=(\Omega, \mathcal{G},\mathcal{G}_t,X_t,\theta_t,\mathbb{P}_x:x\in E,t\in\mathbb{R}_+)$ be a Markov process taking values in $(E,\mathcal{E})$. We assume that a sample space $\Omega$ is such that the usual semi-group of shift operators $(\theta_t)_{t\geq0}$ is well-defined (which is the case, for example, if $\Omega=E^{[0,\infty)}$ is the canonical path space).  If the corresponding semigroup of $X$, $(P_t)$, is the primary object of study, then we say that $X$ is a realisation of a Markov semigroup $(P_t)$. In the case of $(P_t)$ being sub-Markovian, i.e. $P_t1_E\leq1_E$, we extend it to a Markovian semigroup over $E^\Delta=E\cup\{\Delta\}$, where $\Delta$ is a coffin-state. We also denote by $\mathcal{C}(X)=(\Omega, \mathcal{F},\mathcal{F}_t,X_t,\theta_t,\mathbb{P}_x:x\in E,t\in\mathbb{R}_+)$ the \textit{canonical realisation} associated with $X$, defined on $\Omega$ with the filtration $(\mathcal{F}_t)$ deduced from $\mathcal{F}^0_t=\sigma(X_s:s\leq t)$ by standard regularisation procedures (completeness and right-continuity). 

In this paper our standing assumption is that the underlying Markov process $X$ is a \textit{right process} (consult Getoor~\cite{getoor2006markov}, Sharpe~\cite{sharpe1988general} for the general theory). Essentially, right processes are the processes satisfying Meyer's regularity hypotheses (\textit{hypoth\`eses droites}) HD1 and HD2. If a given Markov semigroup $(P_t)$ satisfies HD1 and $\mu$ is an arbitrary probability measure on $(E,\mathcal{E})$, then there exists a homogeneous $E$-valued Markov process $X$ with transition semigroup $(P_t)$ and initial law $\mu$.  Moreover, a realisation of such $(P_t)$ is right-continuous (Sharpe~\cite{sharpe1988general}, Theorem 2.7). Under the second fundamental hypothesis, HD2, $t\to f(X_t)$ is right-continuous for every $\alpha$-excessive function $f$. Recall, for $\alpha>0$, a universally measurable function $f:E\to\mathbb{R}$ is $\alpha$-super-median if $e^{-\alpha t}P_tf\leq f$ for all $t\geq0$, and $\alpha$-excessive if it is $\alpha$-super-median and $e^{-\alpha t}P_tf\to f$ as $t\to0$. If $(P_t)$ satisfies HD1 and HD2 then the corresponding realisation $X$ is strong Markov (Getoor~\cite{getoor2006markov}, Theorem 9.4 and Blumenthal and Getoor~\cite{blumenthal2007markov}, Theorem 8.11).

\begin{remark}
One has the following inclusions among classes of Markov processes:
\begin{equation*}
\text{(Feller)}\subset\text{(Hunt)}\subset\text{(right)}
\end{equation*}
\end{remark}

Let $\mathcal{L}$ be a given extended infinitesimal (martingale) generator of $X$ with a domain $\mathbb{D}(\mathcal{L})$, i.e. we say a Borel function $f:E\to\mathbb{R}$ belongs to $\mathbb{D}(\mathcal{L})$ if there exists a Borel function $h:E\to\mathbb{R}$, such that $\int^t_0\lvert h(X_s)\lvert ds<\infty$, $\forall t\geq0$, $\P_x$-a.s. for each $x$ and the process $M^f=(M^f_t)_{t\geq0}$, given by
\begin{equation}\label{gen}
M^f_t:=f(X_t)-f(x)-\int^t_0h(X_s)ds,\quad t\geq0,\text{ }x\in E,
\end{equation}
is a local martingale under each $\P_x$ (see Revuz and Yor \cite{revuz2013continuous} p.285), and then we write $h=\mathcal{L}f$. 
\begin{remark}Note that if $A\in \mathcal{E}$ and $\P_x(\lambda(\{t:\; X_t\in A\}=0)=1$ for each $x\in E$, where $\lambda$ is Lebesgue measure, then $h$ may be altered on $A$ without affecting the validity of \eqref{gen}, so that, in general, the map $f\to h$ is not unique. This is why we refer to {\em a} martingale generator.
\end{remark}

\paragraph{Optimal stopping problem}
Let $X=(\Omega, \mathcal{G},\mathcal{G}_t,X_t,\theta_t,\mathbb{P}_x:x\in E,t\in\mathbb{R}_+)$ be a right process. Given a function $g:E\to\mathbb{R}$, $\alpha\geq0$ and $T\in\mathbb{R}_+\cup\{\infty\}$ define a corresponding gains process $G^\alpha$ (we simply write $G$ if $\alpha=0$) by $G^\alpha_t=e^{-\alpha t}g(X_t)$ for $t\in[0,T]$. In the case of $T=\infty$, we make the following conventions:
$$
X_{\infty}=\Delta,\quad G^\alpha_\infty=\liminf_{t\to\infty}G^\alpha_t,\quad g(\Delta)=G^0_\infty.
$$
Let $\mathcal{E}^e,\mathcal{E}^u$ be the $\sigma$-algebras on $E$ generated by excessive functions and universally measurable sets, respectively (recall that $\mathcal{E}\subset\mathcal{E}^e\subset\mathcal{E}^u$). We write 
$$g\in\mathcal{Y} \text{, given that } g(\cdot) \text{ is } \mathcal{E}^e \text{-measurable and } G^\alpha \text{ is of class (D).}
$$
For a filtration $(\hat{\mathcal{G}}_t)$, and $(\hat{\mathcal{G}}_t)$ - stopping times $\sigma_1$ and $\sigma_2$, with $\mathbb{P}_x[0\leq \sigma_1\leq \sigma_2\leq T]=1$, $x\in E$,  let $\mathcal{T}_{\sigma_1,\sigma_2}(\hat{\mathcal{G}})$ be the set of $(\hat{\mathcal{G}}_t)$ - stopping times $\tau$ with $\mathbb{P}_x[\sigma_1\leq\tau\leq \sigma_2]=1$. Consider the following optimal stopping problem:
\begin{equation*}
V(x)=\sup_{\tau\in\mathcal{T}_{0,T}(\mathcal{G})}\mathbb{E}_x[e^{-\alpha\tau}g(X_\tau)],\quad x\in E.
\end{equation*}
By convention we set $V(\Delta)=G^\alpha_\infty$. The following result is due to El Karoui et al.~\cite{el1992probabilistic}.
\begin{theorem}\label{thm:VisS}
Let $X=(\Omega, \mathcal{G},\mathcal{G}_t,X_t,\theta_t,\mathbb{P}_x:x\in E,t\in\mathbb{R}_+)$ be a right process with canonical filtration $(\mathcal{F}_t)$. If $g\in\mathcal{Y}$, then
\begin{equation*}
V(x)=\sup_{\tau\in\mathcal{T}_{0,T}(\mathcal{F})}\mathbb{E}_x[e^{-\alpha\tau}g(X_\tau)],\quad x\in E,
\end{equation*}
and $(e^{-\alpha t}V(X_t))$ is a Snell envelope of $G^\alpha$, i.e. for all $x\in E$ and $ \tau\in\mathcal{T}_{0,T}(\mathcal{F})$
\begin{equation*}
e^{-\alpha\tau}V(X_\tau)=\esssup_{\sigma\in\mathcal{T}_{\tau,T}(\mathcal{F})}\mathbb{E}_x[G^\alpha_\sigma\lvert\mathcal{F}_\tau]\quad\mathbb{P}_x\textrm{-a.s.}
\end{equation*}
\end{theorem}
The first important consequence of the theorem is that we can (and will) work with the canonical realisation $\mathcal{C}(X)$. The second one provides a crucial link between the Snell envelope process in the general setting and the value function in the Markovian framework. 
\begin{remark}
The restriction to gains processes of the form $G=g(X)$ (or $G^\alpha$ if $\alpha>0$) is {\em much less restrictive than might appear}. Given that we work on the canonical path space with $\theta$ being the usual shift operator, we can expand the state-space of $X$ by appending an adapted functional $F$, taking values in the space $(E',\mathcal{E}')$, with the property that
\begin{equation}\label{func}
\{F_{t+s}\in A\}\in\sigma(F_s)\cup\sigma(\theta_s\circ X_u:\textit{ }0\leq u\leq t),\quad\textrm{for all }A\in\mathcal{E}'.
\end{equation}
This allows us to deal with time-dependent problems, running rewards and other path-functionals of the underlying Markov process.
\end{remark}
\begin{lemma}\label{Fell}
Suppose $X$ is a canonical Markov process $X$ taking values in the space $(E,\mathcal{E})$  where $E$ is a locally compact, countably based Hausdorff space and $\mathcal{E}$ is its Borel $\sigma$-algebra. Suppose also that $F$ is a path functional of X satisfying \eqref{func} and taking values in the space $(E',\mathcal{E}')$  where $E'$ is a locally compact, countably based Hausdorff space with Borel $\sigma$-algebra $\mathcal{E}'$, then, defining  $Y=(X,F)$, $Y$ is still Markovian. If $X$ is a strong Markov process and $F$ is right-continuous, then $Y$ is strong Markov.
If $X$ is a Feller process and $F$ is right-continuous , then $Y$ is strong Markov,  has a c\`adl\`ag modification and the completion of the natural filtration of $X$, $\mathbb{F}$, is right-continuous and quasi-left continuous, and thus Y is a right process.
\end{lemma}

\begin{example} If $X$ is a one-dimensional Brownian motion, then $Y$, defined by
\begin{equation*}
Y_t=\Bigg(X_t,L^0_t,\sup_{0\leq s\leq t}X_s,\int^t_0\exp(-\int^s_0\alpha(X_u)du)f(X_s)ds\Bigg),\quad t\geq0,
\end{equation*}
where $L^0$ is the local time of $X$ at $0$, is a Feller process on the filtration of $X$.
\end{example}

\section{Main results}\label{main}
In this section we retain the notation of \cref{prelimmart} and \cref{prelimmark}. 
\subsection{General framework}\label{mainmart}
The assumption that $G\in\mathbb{G}$ (i.e. $G$ is a semimartingale with integrable supremum and $G=N+D$ is its canonical decomposition),  neither ensures that $N\in\mathcal{M}_{0}$, nor that $D$ is an $IV$ process, the latter, it turns out, being sufficient for the main result of this section to hold. In order to prove \cref{tabscont} we will need a stronger integrability condition on $G$.

For any adapted c\`{a}dl\`{a}g process $H$, define
\begin{equation}
H^*=\sup_{0\leq t\leq T}\lvert H_t\lvert
\end{equation}
and
\begin{equation}
\lvert\lvert H\lvert\lvert_{\mathcal{S}^p}=\lvert\lvert H^*\lvert\lvert_{L^p}:=\mathbb{E}\big[\lvert H^*\lvert ^p\big]^{1/p},\quad1\leq p\leq\infty.
\end{equation}
\begin{remark}
Note that $\bar{\mathbb{G}}=\mathcal{S}^1$, so that under the current conditions we have that $G\in\mathcal{S}^1$.
\end{remark}
For a special semimartingale $X$ with canonical decomposition $X=\bar{B}+\bar{I}$, where $\bar{B}\in\mathcal{M}_{0,loc}$ and $\bar{I}$ is a predictable $FV$ process, define the $\mathcal{H}^p$ norm, for $1\leq p\leq\infty$, by
\begin{equation}
\lvert\lvert X\lvert\lvert_{\mathcal{H}^p}=\lvert\lvert \bar{B}\lvert\lvert_{\mathcal{S}^p}+\Big\lvert\Big\lvert \int^T_0\lvert d\bar{I}_s\lvert\Big\lvert\Big\lvert_{L^p},
\end{equation}
and, as usual, write $X\in\mathcal{H}^p$ if $\lvert\lvert X\lvert\lvert_{\mathcal{H}^p}<\infty$.
\begin{remark}
A more standard definition of the $\mathcal{H}^p$ norm is with $\lvert\lvert \bar{B}\lvert\lvert_{\mathcal{S}^p}$ replaced by $\lvert\lvert[ \bar{B},\bar{B}]^{1/2}_T\lvert\lvert_{L^p}$. However, the Burkholder-Davis-Gundy inequalities (see Protter \cite{protter2005stochastic}, Theorem 48 and references therein) imply the equivalence of these norms.
\end{remark}
The following lemma follows from the fact that $\bar{I}^*\leq\int^T_0\lvert d\bar{I}_s\vert$, $\mathbb{P}-$a.s:
\begin{lemma}\label{lhp}
On the space of semimartingales, the $\mathcal{H}^p$ norm is stronger than $\mathcal{S}^p$ for $1\leq p<\infty$, i.e. convergence in $\mathcal{H}^p$ implies convergence in $\mathcal{S}^p$.
\end{lemma}
In general, it is challenging to check whether a given process belongs to $\mathcal{H}^1$, and thus the assumption that $G\in\mathcal{H}^1$ might be too stringent. On the other hand, under the assumptions in the Markov setting (see \cref{mainmark}), we will have that $G$ is $locally$ in $\mathcal{H}^1$. Recall that a semimartingale $X$ belongs to $\mathcal{H}^p_{loc}$, for $1\leq p\leq\infty$, if there exists a sequence of stopping times $\{\sigma_n\}_{n\in\mathbb{N}}$, increasing to infinity almost surely, such that for each $n\geq1$, the stopped process $X^{\sigma_n}$ belongs to $\mathcal{H}^p$. Hence, the main assumption in this section is the following:
\begin{assumption}\label{ahp}
$G$ is a semimartingale in both $\mathcal{S}^1$ and $\mathcal{H}^1_{loc}$.
\end{assumption}
\begin{remark}
Given that $G\in\mathcal{H}^1$, \cref{lhp} implies that \cref{ahp} is satisfied, and thus all the results of \cref{prelimmart} hold. Moreover, we then have a canonical decomposition of $G$ 
\begin{equation}\label{gdec}
G=N+D,
\end{equation} 
with $N\in\mathcal{M}_{0,UI}$ and a predictable $IV$ process $D$. On the other hand, under \cref{ahp}, \eqref{gdec} holds only for the stopped process $G^{\sigma_n}$, $n\geq1$.
\end{remark}
We finally arrive to the main result of this section:
\begin{theorem}\label{tabscont}
Suppose \cref{ahp} holds. Let $D^-$ $(D^+)$ denote the decreasing (increasing) components of $D$, as in \cref{ldpredict}. Then $A$ is, as a measure, absolutely continuous with respect to $D^-$ almost surely on $[0,T]$, and $\mu$, defined by
\begin{equation*}
\mu_t:=\frac{dA_t}{dD^-_t},\quad0\leq t\leq T,
\end{equation*}
satisfies $0\leq\mu_t\leq1$ almost surely.
\end{theorem}
\begin{remark}
As is usual in semimartingale calculus, we treat a process of bounded variation and its corresponding Lebesgue-Stiltjes signed measure as synonymous.
\end{remark}
The proof of \cref{tabscont} is based on the discrete-time approximation of the predictable $FV$ processes in the decompositions of $S$ \eqref{DM} and $G$ \eqref{dec}. In particular, let $\mathcal{P}_n=\{0=t^n_0<t^n_1<t^n_2<...<t^n_{k_n}=T\}$, $n=1,2,...$, be an increasing sequence of partitions of $[0,T]$ with $\max_{1\leq k\leq k_n}t^n_k-t^n_{k-1}\to 0$ as $n\to\infty$. Let $S^n_t=S_{t^n_k}$ if $t^n_k\leq t<t^n_{k+1}$ and $S^n_T=S_T$ define the discretizations of $S$, and set
\begin{align}
A^n_t&=0\quad\text{if }0\leq t<t^n_1,\nonumber\\
A^n_t&=\sum^k_{j=1}\mathbb{E}[S_{t^n_{j-1}}-S_{t^n_j}\lvert\mathcal{F}_{t^n_{j-1}}]\quad\text{if }t^n_k\leq t<t^n_{k+1}\text{, }k=1,2,...,k_n-1,\nonumber\\
A^n_T&=\sum^{k_n}_{j=1}\mathbb{E}[S_{t^n_{j-1}}-S_{t^n_j}\lvert\mathcal{F}_{t^n_{j-1}}].\nonumber
\end{align}

If $S$ is \textit{regular} in the sense that for every stopping time $\tau$ and nondecreasing sequence $(\tau_n)_{n\in\mathbb{N}}$ of stopping times with $\tau=\lim_{n\to\infty}\tau_n$, we have $\lim_{n\to\infty}\mathbb{E}[S_{\tau_n}]=\mathbb{E}[S_\tau]$, or equivalently, if $A$ is continuous, Dol\'{e}ans \cite{doleans1968existence} showed that $A^n_t\to A_t$ uniformly in $L^1$ as $n\to\infty$ (see also Rogers and Williams \cite{rogers1987markov}, VI.31, Theorem 31.2). Hence, given that $S$ is regular, we can extract a subsequence $\{A^{n_l}_t\}$, such that $\lim_{l\to\infty}A^{n_l}_t=A_t$ a.s. On the other hand, it is enough for $G$ to be regular:

\begin{lemma}\label{lreg2}
Suppose $G\in\bar{\mathbb{G}}$ is a regular gains process. Then so is its Snell envelope process $S$.
\end{lemma}
See \cref{appx} for the proof.

\begin{remark} If it is not known that $G$ is regular, Kobylanski and Quenez \cite{kobylanski2012optimal}, in a slightly more general setting,  showed that $S$ is still regular, {\em provided} that $G$ is upper semicontinuous in expectation along stopping times, i.e. for all $\tau\in\mathcal{T}^{0,T}$ and for all sequences of stopping times $(\tau_n)_{n\geq1}$ such that $\tau_n\uparrow\tau$, we have
\begin{equation}
\mathbb{E}[G_\tau]\geq\limsup_{n\to\infty}\mathbb{E}[G_{\tau_n}]\nonumber.
\end{equation}
\end{remark}
 
The case where $S$ is not regular is more subtle. In his classical paper Rao \cite{rao1969decomposition} utilised the Dunford-Pettis compactness criterion and showed that, in general, $A^n_t\to A_t$ only \textit{weakly} in $L^1$ as $n\to\infty$ (a sequence $(X_n)_{n\in\mathbb{N}}$ of random variables in $L^1$ converges weakly in $L^1$ to $X$ if for every bounded random variable $Y$ we have that 
$\mathbb{E}[X_nY]\to\mathbb{E}[XY]$ as  $n\to \infty$).

Recall that $weak$ convergence in $L^1$ does not imply convergence in probability, and therefore, we cannot immediately deduce an almost sure convergence along a subsequence. However, it turns out that by modifying the sequence of approximating random variables, the required convergence can be achieved. This has been done in recent improvements of the Doob-Meyer decomposition (see Jakubowski \cite{jakubowski2005almost} and Beiglb\"{o}ck et al. \cite{beiglboeck2012short}. Also, Siorpaes \cite{siorpaes2014approx} showed that there is a subsequence that works for all $(t,\omega)\in[0,T]\times\Omega$ simultaneously). In particular, Jakubowski proceeds as Rao, but then uses Koml\'{o}s's theorem \cite{komlos1967generalization} and proves the following:
\begin{theorem}\label{jakubowski}
There exists a subsequence $\{n_l\}$ such that for $t\in\cup^\infty_{n=1}\mathcal{P}_n$ and as $L\to\infty$
\begin{equation}\label{jakubowskieq}
\frac{1}{L}\Big(\sum^L_{l=1}A^{n_l}_t\Big)\to A_t,\quad\text{a.s. and in }L^1.
\end{equation}
\end{theorem}
 
\begin{proof}[Proof of \cref{tabscont}] Let $(\sigma_n)_{n\geq1}$ be a localising sequence for $G$ such that, for each $n\geq1$, $G^{\sigma_n}=(G_{t\wedge\sigma_n})_{0\leq t\leq T}$ is in $\mathcal{H}^1$. Similarly, set $S^{\sigma_n}=(S_{t\wedge\sigma_n})_{0\leq t\leq T}$ for a fixed $n\geq1$. We need to prove that
\begin{equation}\label{increments}
0\leq A^{\sigma_n}_t-A^{\sigma_n}_s\leq (D^-)^{\sigma_n}_t-(D^-)^{\sigma_n}_s\text{ a.s.},
\end{equation}
since then, as $\sigma_n\uparrow\infty$ almost surely, as $n\to\infty$, and by uniqueness of $A$ and $D^-$, the result follows. In particular, since $A$ is increasing, the first inequality in \eqref{increments} is immediate, and thus we only need to prove the second one. 

After localisation we assume that $G\in\mathcal{H}$. For any $0\leq t\leq T$ and $0\leq\epsilon\leq T-t$ we have that
\begin{align*}
\mathbb{E}[S^{}_{t+\epsilon}\lvert\mathcal{F}_{t}]&=\mathbb{E}\Big[\esssup_{\tau\in\mathcal{T}_{t+\epsilon,T}}\mathbb{E}[G_\tau\lvert\mathcal{F}_{t+\epsilon}]\Big\lvert\mathcal{F}_{t}\Big]\nonumber\\
&\geq\mathbb{E}\Big[\mathbb{E}[G_\tau\lvert\mathcal{F}_{t+\epsilon}]\Big\lvert\mathcal{F}_{t}\Big]\nonumber\\
&=\mathbb{E}[G_\tau\lvert\mathcal{F}_{t}]\text{ a.s.},
\end{align*}
where $\tau\in\mathcal{T}_{t+\epsilon,T}$ is arbitrary. Therefore
\begin{equation}\label{diff}
\mathbb{E}[S^{}_{t+\epsilon}\lvert\mathcal{F}_{t}]\geq\esssup_{\tau\in\mathcal{T}_{t+\epsilon,T}}\mathbb{E}[G_\tau\lvert\mathcal{F}_{t}]\text{ a.s.}
\end{equation}
Then by the definition of $S$ and using \eqref{diff} together with the properties of the $essential$ $supremum$ (see also Lemma \ref{lem:directed} in the \cref{appx}) we obtain
\begin{align}\label{loc1}
\mathbb{E}[S^{}_t-S^{}_{t+\epsilon}\lvert\mathcal{F}_{t}]&\leq\esssup_{\tau\in\mathcal{T}_{t,T}}\mathbb{E}[G_\tau\lvert\mathcal{F}_{t}]-\esssup_{\tau\in\mathcal{T}_{t+\epsilon,T}}\mathbb{E}[G_\tau\lvert\mathcal{F}_{t}]\nonumber\\
&\leq\esssup_{\tau\in\mathcal{T}_{t,T}}\mathbb{E}[G_\tau-G_{\tau\vee(t+\epsilon)}\lvert\mathcal{F}_{t}]\nonumber\\
&=\esssup_{\tau\in\mathcal{T}_{t,t+\epsilon}}\mathbb{E}[G_\tau-G_{\tau\vee(t+\epsilon)}\lvert\mathcal{F}_{t}]\\
&=\esssup_{\tau\in\mathcal{T}_{t,t+\epsilon}}\mathbb{E}[G^{}_\tau-G^{}_{t+\epsilon}\lvert\mathcal{F}_{t}]\text{ a.s.}\nonumber
\end{align}
The first equality in \eqref{loc1} follows by noting that $\mathcal{T}_{t+\epsilon,T}\subset\mathcal{T}_{t,T}$, and that for any $\tau\in\mathcal{T}_{t+\epsilon,T}$ the term inside the expectation vanishes. Using the decomposition of $G^{}$ and by observing that, for all $\tau\in\mathcal{T}_{t,t+\epsilon}$, $(D^+_\tau-D^+_{t+\epsilon})\leq0$, while $N^{}$ is a uniformly integrable martingale, we obtain 
\begin{align}\label{loc2}
\mathbb{E}[S^{}_t-S^{}_{t+\epsilon}\lvert\mathcal{F}_{t}]&\leq\esssup_{\tau\in\mathcal{T}_{t,t+\epsilon}}\mathbb{E}[D^-_{t+\epsilon}-D^-_\tau\lvert\mathcal{F}_{t}]\nonumber\\
&=\mathbb{E}[D^-_{t+\epsilon}-D^-_t\lvert\mathcal{F}_{t}]\text{ a.s.}
\end{align}

Finally, for $0\leq s<t\leq T$, applying Theorem \ref{jakubowski} to $A$ together with \eqref{loc2} gives 
\begin{align}\label{approx}
A_t-A_s&=\lim_{L\to\infty}\frac{1}{L}\Big(\sum_{l=1}^{L}\sum_{j=k'}^{k}\mathbb{E}[S_{t_{j-1}^{n_l}}-S_{t_j^{n_l}}\lvert\mathcal{F}_{t_{j-1}^{n_l}}] \Big)\nonumber\\
&\leq\lim_{L\to\infty}\frac{1}{L}\Big(\sum_{l=1}^{L}\sum_{j=k'}^{k}\mathbb{E}[D^-_{t_{j}^{n_l}}-D^-_{t_{j-1}^{n_l}}\lvert\mathcal{F}_{t_{j-1}^{n_l}}] \Big)\text{ a.s.},
\end{align}
where $k'\leq k$ are such that $t^{n_l}_{k'}\leq s<t^{n_l}_{k'+1}$ and $t^{n_l}_k\leq t<t^{n_l}_{k+1}$ . Note that $D^-$ is also the predictable, increasing $IV$ process in the Doob-Meyer decomposition of the class (D) supermartingale $(G-D^+)$. Therefore we can approximate it in the same way as $A$, so that $D^-_t-D^-_s$ is the almost sure limit along, possibly, a further subsequence $\{n_{l_k}\}$ of $\{n_l\}$, of the right hand side of \eqref{approx}. Here we rely on the special property of the subsequence $\{n_l\}$. In particular, it can be chosen such that convergence \eqref{jakubowskieq} also works along suitable subsequence of any further subsequence, see Remark 1 in Jakubowski \cite{jakubowski2005almost}. 
\end{proof}
We finish this section with a lemma that gives an easy test as to whether the given process belongs to $\mathcal{H}^1_{loc}$ (consult \cref{appx} for the proof).
\begin{lemma}\label{lhloc}
Let $X\in\mathbb{G}$ with a canonical decomposition $X=L+K$, where $L\in\mathcal{M}_{0,loc}$ and $K$ is a predictable $FV$ process. If the jumps of $K$ are uniformly bounded by some finite constant $c>0$, then $X\in\mathcal{H}^1_{loc}$.
\end{lemma}

\subsection{Markovian setting}\label{mainmark}

In the rest of the section (and the paper) we consider the following optimal stopping problem:
\begin{equation}\label{markov}
V(x)=\sup_{\tau\in\mathcal{T}^{0,T}}\mathbb{E}_x[g(X_\tau)],\quad x\in E,
\end{equation}
for a measurable function $g:E\to\mathbb{R}$ and a Markov process $X$ satisfying the following set of assumptions:
\begin{assumption}\label{ass:special}
$X$ is a right process.
\end{assumption}
\begin{assumption}\label{ass:sup}
$\sup_{0 \leq t\leq T}\lvert g(X_t)\lvert\in L^1(\mathbb{P}_x)$, $x\in E$.
\end{assumption}
\begin{assumption}\label{ass:dom}
$g\in\mathbb{D}(\mathcal{L})$, i.e. $g(\cdot)$ belongs to the domain of a  martingale generator of $X$.
\end{assumption}
\begin{remark} \cref{Fell} tells us that if $X$ is Feller and $F$ is an adapted path-functional of the form given in \eqref{func} then (a modification of)  $(X,F)$ satisfies \cref{ass:special}.
\end{remark}
\begin{example}
Let $X=(X_t)_{t\geq0}$ be a Markov process and let $\mathbb{D}(\hat{\mathcal{L}})$ be the domain of a classical infinitesimal generator of $X$, i.e. the set of measurable functions $f:E\to\mathbb{R}$, such that $\lim_{t\to0}(\mathbb{E}_x[f(X_t)]-f(x))/t$ exists. Then $\mathbb{D}(\mathcal{\hat{L}})\subset\mathbb{D}(\mathcal{L})$. 
In particular, 
\begin{itemize}
\item[1.] if $X=(X_t)_{t\geq0}$ is a solution of an SDE driven by a Brownian motion in $\mathbb{R}^d$, then $C^2\subset\mathbb{D}(\mathcal{\hat{L}})$;
\item[2.] if the state space $E$ is finite (so that $X$ is a continuous time Markov chain), then any measurable and bounded $f:E\to\mathbb{R}$ belongs to $\mathbb{D}(\mathcal{\hat{L}})$
\item[3.] if $X$ is a L\'evy process on ${\mathbb R}^d$ with finite variance increments then $C^2({\mathbb R}^d,{\mathbb R})\subset \mathbb{D}(\mathcal{\hat{L}})$
 \end{itemize}

 \end{example}
 Note that the gains process is of the form $G=g(X)$, while by \cref{thm:VisS}, the corresponding Snell envelope is given by
\begin{equation*}
S_t^T:=\begin{cases}
V(X_t):t< T,\\
g(X_T):t\geq T.
\end{cases}
\end{equation*}
In a similar fashion to that in the general setting, \cref{ass:sup} ensures the class (D) property for the gains and Snell envelope processes. Moreover, under \cref{ass:dom},
\begin{equation}\label{gendec}
g(X_t)=g(x)+M^g_t+\int^t_0\mathcal{L}g(X_s)ds,\quad0\leq t\leq T,\,x\in E,
\end{equation}
and the $FV$ process in the semimartingale decomposition of $G=g(X)$ is absolutely continuous with respect to Lebesgue measure, and therefore predictable, so that \eqref{gendec} is a canonical semimartingale decomposition of $G=g(X)$. Then, by \cref{ass:sup}, and using \cref{lhloc}, we also deduce that $g(X)\in\mathcal{H}^1_{loc}$.
\begin{remark}
When $T<\infty$, the optimal stopping problem, in general, is time-inhomogeneous, and we need to replace the process $X_t$ by the process $Z_t=(t,X_t)$, $t\in[0,T]$, so that \eqref{markov} reads
\begin{equation}\label{markov1}
\tilde{V}(t,x)=\sup_{\tau\in\mathcal{T}_{0,T-t}}\mathbb{E}_{t,x}[\tilde{g}(t+\tau,X_{t+\tau})],\quad x\in E,
\end{equation}
where $\tilde{g}:[0,T]\times E\to\mathbb{R}$ is a new payoff function (consult Peskir and Shiryaev \cite{peskir2006optimal} for examples). In this case, \cref{ass:dom} should be replaced by a requirement that there exists a measurable function $\tilde{h}:[0,T]\times E\to\mathbb{R}$ such that $M^{\tilde{g}}_t:=\tilde{g}(Z_t)-\tilde{g}(0,x)-\int^t_0\tilde{h}(Z_s)ds$ defines a local martingale.
\end{remark}

The crucial result of this section is the following: 
\begin{theorem}\label{pvdom}
Suppose Assumptions \ref{ass:special}, \ref{ass:sup} and \ref{ass:dom} hold. Then $V\in\mathbb{D}(\mathcal{L})$.
\end{theorem}
\begin{proof}
In order to be consistent with the notation in the general framework, let
\begin{equation*}
D_t:=g(X_0)+\int^t_0\mathcal{L}g(X_s)ds,\quad0\leq t\leq T.
\end{equation*}
Recall \cref{ldpredict}. Then $D^+$ and $D^-$ are explicitly given (up to initial values) by
\begin{align*}
D^+_t:&=\int^t_0\mathcal{L}g(X_s)^+ds,\nonumber\\
D^-_t:&=\int^t_0 \mathcal{L}g(X_s)^-ds.
\end{align*}
In particular, $D^-$ is, as a measure, absolutely continuous with respect to Lebesgue measure. By applying \cref{tabscont}, we deduce that 
\begin{equation}\label{dec2}
V(X_t)=V(x)+M^*_t-\int_0^t\mu_s\mathcal{L}g(X_s)^-ds,\quad 0\leq t\leq T,\; x\in\mathbb{R},
\end{equation}
where $\mu$ is a non-negative Radon-Nikodym derivative with $0\leq\mu_s\leq1$. Then we also have that $\int^t_0\lvert\mu_s \mathcal{L}g(X_s)^-\lvert ds<\infty$, for every $0\leq t\leq T$.

In order to finish the proof we are left to show that there exists a suitable measurable function $\lambda:E\to\mathbb{R}$ such that $A_t=\int^t_0\mu_s\mathcal{L}g(X_s)^- ds=\int^t_0\lambda(X_s)ds$ a.s., for all $t\in[0,T]$. For this, recall that a process $Z$ (on $(\Omega, \mathcal{G},\mathcal{G}_t,X_t,\theta_t,\mathbb{P}_x:x\in E,t\in\mathbb{R}_+)$ or just on $\mathcal C(X)$) is
 \textit{additive} if $Z_0=0$ a.s. and $Z_{t+s}=Z_t+Z_s\circ\theta_t$ a.s., for all $s,t\in[0,T]$. Moreover, for any measurable function $f:E\to\mathbb{R}$, $Z^f_t=f(X_t)-f(x)$ defines an additive process. In particular, if $Z^f$ is also a semimartingale, then the martingale and $FV$ processes in the decomposition of $Z^f$ are also additive (\c{C}inlar et al. \cite{cinlar1980} gives necessary and sufficient conditions for $Z^f$ to be a semimartingale). 

Finally, we have that $A_t=\int^t_0\mu_s\mathcal{L}g(X_s)^-ds$, $t\in[0,T]$, is an increasing additive process such that $dA_t \ll dt$. Set $K_t=\liminf_{s\downarrow0,s\in\mathbb{Q}}(A_{t+s}-A_t)/s$ and $\beta(x)=\mathbb{E}_x[K_0]$, $x\in E$. Then by Proposition 3.56 in \c{C}inlar et al. \cite{cinlar1980}, we have that, for $t\in[0,T]$, $A_t=\int^t_0\beta(X_s)ds$ $\mathbb P_x$-a.s. for each $x\in E$.
\end{proof}
\begin{remark}
In some specific examples it is possible to relax \cref{ass:dom}. Let $\mathcal{S}:=\{x\in E: V(x)=g(x)\}$ be the stopping region. It is well-known that $S=V(X)$ is a martingale on the go region $\mathcal{S}^c$, i.e. $M^c$ given by
$$
M^c_t\defto\int_0^t 1_{(X_{s-}\in \S^c)}dS_s
$$
is a martingale (see \cref{lreg1}). 
This implies that $\int_0^t1_{(X_{s-}\in \S^c)}dA_s=0$, and therefore we note that in order for $V\in\mathbb{D}(\mathcal{L})$, we need $D$ to be absolutely continuous with respect to Lebesgue measure $\lambda$ only on the stopping region i.e. that $\int_0^{\cdot}1_{(X_{s-}\in \S)}dD_s\ll \lambda$. For example, let $E=\mathbb{R}$, fix $K\in\mathbb{R}_+$ and consider $g(\cdot)$ given by $g(x)=(K-x)^+$, $x\in E$. We can easily show, under very weak conditions, that $\mathcal{S}\subset [0,K]$ and so we need only have that $\int_0^{\cdot}1_{(X_{s-}<K)}dD_s$ is absolutely continuous.
\end{remark}

\section{Applications: duality, smooth fit}\label{app}
In this section we retain the setting of \cref{mainmark}.
\subsection{Duality}\label{duality}
Let $x\in E$ be fixed. As before, let $\mathcal M_{0,UI}^x$ denote all the right-continuous uniformly integrable c\`{a}dl\`{a}g martingales (started at zero) on the filtered space $(\Omega, \mathcal F,\mathbb F,\P_x)$, $x\in E$.
The main result of Rogers \cite{rogers2002monte} in the Markovian setting reads:
\begin{theorem}\label{tdual}
Suppose Assumptions \ref{ass:special} and \ref{ass:sup} hold. Then
\begin{equation}\label{gendual}
V(x)=\sup_{\tau\in\mathcal{T}^{0,T}}\mathbb{E}_x[G_\tau]=\inf_{M\in\mathcal{M}_{0,UI}^x}\mathbb{E}_x\Big[\sup_{0\leq t\leq T}\Big(G_t-M_t\Big)\Big],\quad x\in E.
\end{equation}
\end{theorem}

We call the right hand side of \eqref{gendual} the $dual$ of the optimal stopping problem. In particular, the right hand side of \eqref{gendual} is a "generalised stochastic control problem of Girsanov type", where a controller is allowed to choose a martingale from $\mathcal M_{0,UI}^x$, $x\in E$. Note that an optimal martingale for the dual is $M^*$, the martingale appearing in the Doob-Meyer decomposition of $S$, while any other martingale in $\mathcal{M}_{0,UI}^x$ gives an upper bound of $V(x)$. We already showed that $M^*=M^V$, which means that, when solving the dual problem, one can search only over martingales of the form $M^f$, for $f\in\mathbb{D}(\mathcal{L})$, or equivalently over the functions $f\in\mathbb{D}(\mathcal{L})$. We can further define $\mathcal{D}_{\mathcal{M}_{0,UI}}\subset\mathbb{D}(\mathcal{L})$ by
\begin{equation*}
\mathcal{D}_{\mathcal{M}_{0,UI}}:=\{f\in\mathbb{D}(\mathcal{L}): f\geq g,f\textrm{ is superharmonic, } M^f\in\mathcal{M}_{0,UI}\}.
\end{equation*}
To conclude that $V\in\mathcal{D}_{\mathcal{M}_{0,UI}}$ we need to show that $V$ is superharmonic, i.e. for all stopping times $\sigma\in\mathcal{T}^{0,T}$ and all $x\in E$, $\mathbb{E}_x[V(X_\sigma)]\leq V(x)$. But this follows immediately from the Optional Sampling theorem, since $S=V(X)$ is a uniformly integrable supermartingale. Hence, as expected, we can restrict our search for the best minimising martingale to the set $\mathcal{D}_{\mathcal{M}_{0,UI}}$.

\begin{theorem}\label{tmarkov}
The dual problem, i.e. the right hand side of \eqref{gendual}, is a stochastic control problem for a controlled Markov process when $G=g(X)$ and the assumptions of \cref{pvdom} hold. 
\end{theorem}
\begin{proof}
 For any $f\in\mathcal{D}_{\mathcal{M}^x_{0,UI}}$, $x\in E$ and $y,z\in\mathbb{R}$, define processes $Y^f$ and $Z^f$ via
\begin{align*}
Y^f_t&:=y+\int^t_0\mathcal{L}f(X_s)ds,\quad0\leq t\leq T,\\
Z^f_{s,t}&:=\sup_{s\leq r \leq t}\Big(f(x)+g(X_r)-f(X_r)+Y^f_r\Big),\quad 0\leq s\leq t\leq T,
\end{align*}
and to allow arbitrary starting positions, set $Z^f_{t}=Z^f_{0,t}\vee z$, for $z\geq g(x)+y$. Note that, for any $f\in\mathbb{D}(\mathcal{L})$, $Y^f$ is an additive functional of $X$. \cref{Fell} implies that
if $f\in\mathcal{D}_{\mathcal{M}_{0,UI}}$ then $(X,Y^f,Z^f)$ is a Markov process.

Define $\hat{V}:E\times\mathbb{R}^2\to\mathbb{R}$ by
\begin{equation*}
\hat{V}(x,y,z)=\inf_{f\in\mathcal{D}_{\mathcal{M}^x_{0,UI}}}\mathbb{E}_{x,y,z}[Z^f_{T}],\quad (x,y,z)\in E\times\mathbb{R}\times\mathbb{R}.
\end{equation*} 
It is clear that this is a stochastic control problem for the controlled Markov process $(X,Y^f,Z^f)$, where the admissible controls are functions in $\mathcal{D}_{\mathcal{M}_{0,UI}}$. Moreover, since $V\in\mathcal{D}_{\mathcal{M}_{0,UI}}$, by virtue of \cref{tdual}, and adjusting initial conditions as necessary, we have
\begin{equation*}
V(x)=\hat{V}(x,0,g(x))=\mathbb{E}_{x,0,g(x)}[Z^V_{T}],\quad x\in E.
\end{equation*}
{\color{white}a}
\end{proof}

\subsection{Some remarks on the \textit{smooth pasting} condition}\label{smoothpasting}
We will now discuss the implications of \cref{pvdom} for the smoothness of the value function $V(\cdot)$ of the optimal stopping problem given in \eqref{markov}.
\begin{remark}
While in \cref{tpasting} (resp. \cref{tKilledpasting}) we essentially recover (a small improvement of) Theorem 2.3 in Peskir \cite{peskir2007principle} (resp. Theorem 2.3 in Samee \cite{samee2010fit}), the novelty is that we prove the results by means of stochastic calculus, as opposed to the analytic approach in \cite{peskir2007principle} (resp. \cite{samee2010fit}).
\end{remark}

In addition to Assumptions \ref{ass:sup} and \ref{ass:dom}, we now assume that $X$ is a one-dimensional diffusion in the It\^{o}-McKean \cite{ito1965diff} sense, so that $X$ is a strong Markov process with continuous sample paths. We also assume that the state space $E\subset\mathbb{R}$ is an interval with endpoints $-\infty\leq a\leq b\leq+\infty$. Nnote that the diffusion assumption implies \cref{ass:special}. Finally, we assume that $X$ is $regular$: for any $x,y\in$ int$(E)$, $\mathbb{P}_x[\tau_y<\infty]>0$, where $\tau_y=\min\{t\geq0:X_t=y\}$. Let $\alpha\geq0$ be fixed;  $\alpha$ corresponds to a killing rate of the sample paths of $X$.

 \paragraph{The case without killing: $\alpha=0$}Let $s(\cdot)$ denote a scale function of $X$, i.e. a continuous, strictly increasing function on $E$ such that for $l$, $r$, $x\in E$, with $a\leq l<x<r\leq b$, we have 
\begin{equation}\label{scale}
\mathbb{P}_x(\tau_r<\tau_l)=\frac{s(x)-s(l)}{s(r)-s(l)},
\end{equation}
see Revuz and Yor \cite{revuz2013continuous}, Proposition 3.2 (p.301) for the proof of existence and properties of such a function.

From \eqref{scale}, using regularity of $X$ and that $V(X)$ is a supermartingale of class (D) we have that $V(\cdot)$ is $s$-concave:
\begin{align}
V(x)\geq V(l)\frac{s(r)-s(x)}{s(r)-s(l)}+V(r)\frac{s(x)-s(l)}{s(r)-s(l)},\quad x\in[l,r].
\end{align}

\begin{theorem}\label{tpasting}
Suppose the assumptions of \cref{pvdom} are satisfied, so that $V\in\mathbb{D}(\mathcal{L})$. Further assume that $X$ is a regular, strong Markov process with continuous sample paths.  Let $Y=s(X)$, where $s(\cdot)$ is a scale function of $X$. 

\begin{enumerate}
\item Assume that for each $y\in [s(a),s(b)]$, the local time of $Y$ at $y$, $L^y$, is singular with respect to Lebesgue measure. Then, if $s\in{C}^1$, $V(\cdot)$, given by \eqref{markov}, belongs to ${C}^1$.

\item Assume that $([Y,Y]_t)_{t\geq0}$ is, as a measure, absolutely continuous with respect to Lebesgue measure. If $s^{\prime}(\cdot)$ is absolutely continuous, then $V\in C^1$ and $V^\prime(\cdot)$ is also absolutely continuous. 
\end{enumerate}
\end{theorem}

\begin{remark}\label{useful}
If $\mathcal G$ is the filtration of a  Brownian motion, $B$, then $Y=s(X)$ is a stochastic integral with respect to $B$ (a consequence of martingale representation):
\begin{equation}
Y_t=Y_0+\int_0^t\sigma_s dB_s.
\end{equation}
Moreover,  Proposition 3.56 in \c{C}inlar et al. \cite{cinlar1980} ensures that $\sigma_t=\sigma(Y_t)$ for a suitably measurable function $\sigma$ and 
$$
[Y,Y]_t=\int_0^t \sigma^2(Y_s)ds.
$$

 In this case, both, the singularity of the local time of Y and absolute continuity of $[Y,Y]$ (with respect to Lebesgue measure), are inherited from those of Brownian motion. On the other hand, if X is a regular diffusion (not necessarily a solution to an SDE driven by a Brownian motion), absolute continuity of $[Y,Y]$ still holds, if the speed measure of $X$ is absolutely continuous (with respect to Lebesgue measure).
\end{remark}
\begin{proof}

Note that $Y=s(X)$ is a Markov process, and let $\mathcal{K}$ denote its martingale generator. Moreover, $V(x)=W(s(x))$ (see \cref{lemW} and the following remark), where, on the interval $[s(a),s(b)]$, $W(\cdot)$ is the smallest nonnegative concave majorant of the function $\hat{g}(y)= g\circ s^{-1}(y)$. Then, since $V\in\mathbb{D}(\mathcal{L})$,
\begin{equation}\label{rep2}
V(X_t)=V(x)+M^V_t+\int^t_0\mathcal{L}V(X_u)du,\quad\textit{ }0\leq t\leq T,\nonumber\\
\end{equation}
and thus
\begin{align*}
W(Y_t)&=W(y)+M^V_t+\int^t_0(\mathcal{L}V)\circ s^{-1}(Y_u)du,\quad\textit{ }0\leq t\leq T.
\end{align*}
Therefore, $W \in\mathbb{D}(\mathcal{K})$, since
\begin{equation}\label{w1}
W(Y_t)=W(y)+M^V_t+\int^t_0\mathcal{K}W(Y_u)du,
\end{equation}
for $y\in[s(a),s(b)]$, $0\leq t\leq T$, with $\mathcal{K}W=\mathcal{L}V\circ s^{-1}\leq 0$.

On the other hand, using the generalised It\^{o} formula for concave/convex functions (see e.g. Revuz and Yor \cite{revuz2013continuous}, Theorem 1.5 p.223) we have
\begin{equation*}
W(Y_t)=W(y)+\int^t_0W^{'}_{+}(Y_u)dY_u-\int^{s(b)}_{s(a)}L_t^z\nu(dz),
\end{equation*}
for $y\in[s(a),s(b)]$, $0\leq t\leq T$, where $L^z_t$ is the local time of $Y_t$ at $z$, and $\nu$ is a non-negative $\sigma$-finite measure corresponding to the second derivative of $-W$ in the sense of distributions. Then, by the uniqueness of the decomposition of a special semimartingale, we have that, for $t\in[0,T]$,
\begin{equation}\label{eqintegrals}
-\int^t_0\mathcal{K}W(Y_u)du=\int^{s(b)}_{s(a)}L_t^z\nu(dz)\quad\textrm{a.s.}
\end{equation}

In order to prove the first claim, using the Lebesgue decomposition theorem, split $\nu$ into $\nu=\nu_c+\nu_s$, where $\nu_c$ and $\nu_s$ are measures, absolutely continuous and singular (with respect to Lebesgue measure), respectively, so that 
\begin{align}\label{w2}
\int^{s(b)}_{s(a)}L_t^z\nu(dz)&=\int^{s(b)}_{s(a)}L_t^z\nu'_c(z)dz+\int^{s(b)}_{s(a)}L_t^z\nu_s(dz)\quad\textrm{a.s.}
\end{align}
 Now suppose that $\nu_s(\{z_0\})>0$ for some $z_0\in(s(a),s(b))$. Then, using \eqref{eqintegrals} and \eqref{w2}, we have 
\begin{equation}\label{eq2int}
-\int^t_0\mathcal{K}W(Y_u)du=\int^{t}_{0}L^z_t\nu^\prime_c(z)dz+
L^{z_0}_t\nu_s(\{z_0\})+\int^{s(b)}_{s(a)}\mathbbm{1}_{\{z\neq z_0\}}L_t^{z}\nu_s(dz)\quad\textrm{a.s.}
\end{equation}
Since $L_t^{z_0}$ is positive with positive probability and, by assumption, $L^y$, $y\in[s(a),s(b)]$, is singular with respect to Lebesgue measure, the right hand side of \eqref{eq2int} contradicts absolute continuity of the left hand side. Therefore, $\nu_s(\{z_0\})=0$, and since $z_0$ was arbitrary, we have that $\nu_s$ does not charge points (so that $\nu_s$ is singular continuous with respect to Lebesgue measure). It follows that $W\in C^1$. Since $s\in{C}^1$ by assumption, we conclude that $V\in{C}^1$.

We now prove the second claim. By assumption, $[Y,Y]$ is absolutely continuous with respect to Lebesgue measure (on the time axis). 
Invoking Proposition 3.56 in \c{C}inlar et al. \cite{cinlar1980} again, we have that
$$
[Y,Y]_t=\int\sigma^2(Y_u)du
$$
(as in Remark \ref{useful}). 
 A time-change argument allows us to conclude that $Y$ is a time-change of a BM and that we may neglect the set $\{t:\sigma^2(Y_t)=0\}$ in the representation (\ref{rep2}). Thus 
 $$
 W(Y_t)=W(Y_0)+\int_0^t 1_{N^c}(Y_u)dM^V_u+\int_0^t 1_{N^c}(Y_u)\mathcal{K}W(Y_u)du
 $$
 where  $N$ is the zero set of $\sigma$. 
Then, using the occupation time formula (see, for example, Revuz and Yor \cite{revuz2013continuous}, Theorem 1.5 p.223) we have that
\begin{equation*}
-\int^t_0\mathcal{K}W(Y_u)du=\int^t_0f(Y_u)d[Y,Y]_u=\int^{s(b)}_{s(b)}f(z)L_t^zdz\quad\textrm{a.s.,}
\end{equation*}
where $f:[s(a),s(b)]\to\mathbb{R}$ is given by $f:y\mapsto - \frac{\mathcal{K}W}{\sigma}1_{N^c}(y)$. Now observe that, for $0\leq r\leq t\leq T$, $\eta([r,t]):=\int^{s(b)}_{s(a)}f(z)\Big(L_t^z-L_r^z\Big)dz$ and $\pi([r,t]):=\int^{s(b)}_{s(a)}\Big(L_t^z-L_r^z\Big)\nu(dz)$ define measures on the time axis, which, by virtue of \eqref{eqintegrals}, are equal (and thus both are absolutely continuous with respect to Lebesgue measure). Now define $T^{\underline l,\bar l}:=\{t:Y_t\in[\underline l,\bar l]\}$, $s(a)\leq \underline l\leq \bar l\leq s(b)$. Then the restrictions of $\eta$ and $\pi$ to $T^{\underline l,\bar l}$, $\eta\lvert_{T^{\underline l,\bar l}}$ and $\pi\lvert_{T^{\underline l,\bar l}}$, are also equal. Moreover, since $Y$ is a local martingale, it is also a semimartingale. Therefore, for every $0\leq t\leq T$, $L_t^z$ is carried by the set $\{t:Y_t=z\}$ (see Protter \cite{protter2005stochastic}, Theorem 69 p.217). Hence, for each $t\in[0,T]$,
\begin{equation}\label{eq3int}
\eta\lvert_{T^{\underline l,\bar l}}([0,t])=\int_{\underline l}^{\bar l}L^z_tf(z)dz=\int_{\underline l}^{\bar l}L^z_t\nu(dz)=\pi\lvert_{T^{\underline l,\bar l}}([0,t]),
\end{equation}
and, since $\underline l$ and $\bar l$ are arbitrary, the left and right hand sides of \eqref{eq3int} define measures on $[s(a),s(b)]\subseteq\mathbb{R}$, which are equal. It follows that $f(z)dz=\nu(dz)$, and thus, by uniqueness of the Lebesgue decomposition of $\sigma$-finite measures, $\nu_s=0$. This proves that $W\in C^1$ and $W^\prime(\cdot)$ is absolutely continuous on $[s(a),s(b)]$ with Radon-Nykodym derivative $f$. Since the product and composition of absolutely continuous functions are absolutely continuous, we conclude that $V^\prime(\cdot)$ is absolutely continuous (since $s^\prime(\cdot)$ is, by assumption).
 \end{proof}
\begin{remark}
We note that for a smooth fit principle to hold, it is not necessary that $s\in{C}^1$. Given that all the other conditions of \cref{tpasting} hold, it is sufficient that $s(\cdot)$ is differentiable at the boundary of the continuation region. On the other hand, if $g\in\mathbb{D}(\mathcal{L})$, $V\in C^1$, even if $g\notin C^1$.

Moreover, since $V=g$ on the stopping region, \cref{tpasting} tells us that $g\in C^1$ on the interior of the stopping region. However, the question whether this stems already from the assumption that $g\in\mathbb{D}(\mathcal{L})$ is more subtle. For example, if $g\in\mathbb{D}(\mathcal{L})$ and $g$ is a difference of two convex functions, then by the generalised It\^{o} formula and the local time argument (similarly to the proof of \cref{tpasting}) we could conclude that $g\in C^1$ on the whole state space $E$.
\end{remark}

\paragraph{Case with killing: $\alpha>0$} We now generalise the results of the \cref{tpasting} in the presence of a non-trivial killing rate. Consider the following optimal stopping problem
\begin{equation}\label{ValueKill}
V(x)=\sup_{\tau\in\mathcal{T}^{0,T}}\mathbb{E}_x[e^{-\alpha\tau}g(X_\tau)],\quad x\in E.
\end{equation}
Note that, since $\alpha>0$, using the regularity of $X$ together with the supermartingale property of $V(X)$ we have that
\begin{align}\label{Vconcave}
V(x)\geq V(l)\mathbb{E}_x[e^{-\alpha\tau_l}1_{\tau_l<\tau_r}]+V(r)\mathbb{E}_x[e^{-\alpha\tau_r}1_{\tau_r<\tau_l}],\quad x\in[l,r]\subseteq E.
\end{align}Define increasing and decreasing functions $\psi,\phi:E\to\mathbb{R}$, respectively, by
\begin{align}\label{psi_phi}
  \psi(x) = \begin{cases}
      \mathbb{E}_x[e^{-\alpha\tau_c}], & \text{if $x\leq c$} \\
      1/\mathbb{E}_c[e^{-\alpha\tau_x}] ,& \text{if $x>c$}
    \end{cases}\quad
    \phi(x) = \begin{cases}
     1/\mathbb{E}_c[e^{-\alpha\tau_x}] , & \text{if $x\leq c$} \\
      \mathbb{E}_x[e^{-\alpha\tau_c}] ,& \text{if $x>c$}
    \end{cases}
\end{align}
where $c\in E$ is arbitrary. Then, $(\Psi_t)_{0\leq t\leq T}$ and $(\Phi_t)_{0\leq t\leq T}$, given by
\begin{equation*}\label{eq:PsiPhi}
\Psi_t=e^{-\alpha t}\psi(X_t),\quad\Phi_t=e^{-\alpha t}\phi(X_t),\quad0\leq t\leq T,
\end{equation*}
respectively, are local martingales (and also supermartingales, since $\psi,\phi$ are non-negative); see Dynkin \cite{dynkin1965markov} and It\^{o} and McKean \cite{ito1965diff}.

Let $p_1,p_2:[l,r]\to[0,1]$ (where $[l,r]\subseteq E$) be given by
\begin{equation*}
p_1(x)=\mathbb{E}_x[e^{-\alpha\tau_l}1_{\tau_l<\tau_r}],\quad p_2(x)=\mathbb{E}_x[e^{-\alpha\tau_r}1_{\tau_r<\tau_l}].
\end{equation*}
Continuity of paths of $X$ implies that $p_i(\cdot),i=1,2$, are both continuous (the proof of continuity of the scale function in \eqref{scale} can be adapted for a killed process). In terms of the functions $\psi(\cdot)$, $\phi(\cdot)$ of \eqref{psi_phi}, using appropriate boundary conditions, one calculates
\begin{equation}\label{p1_p2}
p_1(x)=\frac{\psi(x)\phi(r)-\psi(r)\phi(x)}{\psi(l)\phi(r)-\psi(r)\phi(l)},\quad p_2(x)=\frac{\psi(l)\phi(x)-\psi(x)\phi(l)}{\psi(l)\phi(r)-\psi(r)\phi(l)},\quad x\in[l,r].
\end{equation}
Let $\tilde{s}:E\to\mathbb{R}_+$ be the continuous increasing function defined by $\tilde{s}(x)=\psi(x)/\phi(x)$. Substituting \eqref{p1_p2} into \eqref{Vconcave} and then dividing both sides by $\phi(x)$ we get
\begin{equation*}
\frac{V(x)}{\phi(x)}\geq\frac{V(l)}{\phi(l)}\cdot\frac{\tilde{s}(r)-\tilde{s}(x)}{\tilde{s}(r)-\tilde{s}(l)}+\frac{V(r)}{\phi(r)}\cdot\frac{\tilde{s}(x)-\tilde{s}(l)}{\tilde{s}(r)-\tilde{s}(l)},\quad x\in[l,r]\subseteq E,
\end{equation*}
so that $V(\cdot)/\phi(\cdot)$ is $\tilde{s}$-concave.

Recall that \cref{Vconcave} essentially follows from $V(\cdot)$ being $\alpha$-superharmonic, so that it satisfies $\mathbb{E}_x[e^{-\alpha\tau}V(X_\tau)]\leq V(x)$ for $x\in E$ and any stopping time $\tau$. Since $\Phi$ and $\Psi$ are local martingales, it follows that the converse is also true, i.e. given a measurable function $f:E\to\mathbb{R}$, $f(\cdot)/\phi(\cdot)$ is $\tilde{s}$-concave if and only if $f(\cdot)$ is $\alpha$-superharmonic (Dayanik and Karatzas \cite{dayanik2003optimal}, Proposition 4.1). This shows that a value function $V(\cdot)$ is the minimal majorant of $g(\cdot)$ such that $V(\cdot)/\phi(\cdot)$ is $\tilde{s}$-concave.

\begin{lemma}\label{lemW}
Suppose $[l,r]\subseteq E$ and let $W(\cdot)$ be the smallest nonnegative concave majorant of $\tilde{g}:=(g/\phi)\circ \tilde{s}^{-1}$ on $[\tilde{s}(l),\tilde{s}(r)]$, where $\tilde{s}^{-1}$ is the inverse of $\tilde{s}$. Then $V(x)=\phi(x)W(\tilde{s}(x))$ on $[l,r]$.
\end{lemma}
\begin{proof}
Define $\hat{V}(x)=\phi(x)W(\tilde{s}(x))$ on $[l,r]$. Then, trivially, $\hat{V}(\cdot)$ majorizes $g(\cdot)$ and $\hat{V}(\cdot)/\phi(\cdot)$ is $\tilde{s}$-concave. Therefore $V(x)\leq\hat{V}(x)$ on $[l,r]$.

On the other hand, let $\hat{W}(y)=(V/\phi)(\tilde{s}^{-1}(y))$ on $[\tilde{s}(l),\tilde{s}(r)]$. Since $V(x)\geq g(x)$ and $(V/\phi)(\cdot)$ is $\tilde{s}$-concave on $[l,r]$, $\hat{W}(\cdot)$ is concave and majorizes $(g/\phi)\circ \tilde{s}^{-1}(\cdot)$ on $[\tilde{s}(l),\tilde{s}(r)]$. Hence, $W(y)\leq\hat{W}(y)$ on $[\tilde{s}(l),\tilde{s}(r)]$.

Finally, $(V/\phi)(x)\leq(\hat{V}/\phi)(x)=W(\tilde{s}(x))\leq\hat{W}(\tilde{s}(x))=(V/\phi)(x)$ on $[l,r]$.
\end{proof}
\begin{remark}
When $\alpha=0$, let $(\psi,\phi)=(s,1)$. Then \cref{lemW} is just Proposition 4.3. in Dayanik and Karatzas~\cite{dayanik2003optimal}.
\end{remark}

With the help of \cref{lemW} and using parallel arguments to those in the proof of \cref{tpasting} we can formulate sufficient conditions for $V$ to be in ${C}^1$ and have absolutely continuous derivative. 
\begin{theorem}\label{tKilledpasting}
Suppose the assumptions of \cref{pvdom} are satisfied, so that $V\in\mathbb{D}(\mathcal{L})$. Further assume that $X$ is a regular Markov process with continuous sample paths. Let $\psi(\cdot),\phi(\cdot)$ be as in \eqref{psi_phi} and consider the process $Y=\tilde{s}(X)$. 
\begin{enumerate}
\item Assume that, for each $y\in [\tilde{s}(a),\tilde{s}(b)]$, the local time of $Y$ at $y\in [\tilde{s}(a),\tilde{s}(b)]$, $\hat{L}^y$, is singular with respect to Lebesgue measure. Then if $\psi,\phi\in{C}^1$, $V(\cdot)$, given by \eqref{ValueKill}, belongs to ${C}^1$.

\item Assume that $[Y,Y]$ is, as a measure, absolutely continuous with respect to Lebesgue measure. If $\psi^\prime(\cdot),\phi^\prime(\cdot)$ are both absolutely continuous, then $V^\prime(\cdot)$ is aslo absolutely continuous.
\end{enumerate}
\end{theorem}
\begin{proof}
First note that $Y$ is not necessarily a local martingale, while  $\Phi Y$ is. Indeed, $\Phi Y=\Psi$. Hence
\begin{equation*}
(N_t)_{0\leq t\leq T}:=\Big(\int^t_0\Phi_tdY_t+[\Phi,Y]_t\Big)_{0\leq t\leq T}
\end{equation*}
is the difference of two local martingales, and thus is a local martingale itself. Using the generalised It\^{o} formula for concave/convex functions, we have 
\begin{equation}\label{eqWconc}
\Phi_tW(Y_t)=\Phi_0W(y)+\int^t_0W(Y_s)d\Phi_s+\int^t_0W^{'}_{+}(Y_s)dN_s-\int^{\tilde{s}(r)}_{\tilde{s}(l)}\Phi_t\hat{L}_t^z\nu(dz),
\end{equation}
for $y\in[\tilde{s}(l),\tilde{s}(r)]$, $0\leq t\leq T$, where $\hat{L}^z_t$ is the local time of $Y_t$ at $z$, and $\nu$ is a non-negative $\sigma$-finite measure corresponding to the derivative $W^{''}$ in the sense of distributions.

On the other hand, if $g\in\mathbb{D}(\mathcal{L})$, then $V\in\mathbb{D}(\mathcal{L})$. Therefore,
\begin{equation}\label{eqVkill}
e^{-\alpha t}V(X_t)=V(x)+\int_0^te^{-\alpha s}dM^V_s+\int^t_0e^{-\alpha s}\{\mathcal{L}-\alpha\}V(X_s)ds,\quad\textit{ }0\leq t\leq T.
\end{equation}
Then, similarly to before, from the uniqueness of the decomposition of the Snell envelope, we have that the martingale and $FV$ terms in \eqref{eqWconc} and \eqref{eqVkill} coincide. Hence, for $t\in[0,T]$,
\begin{equation*}
\int^{\tilde{s}(r)}_{\tilde{s}(l)}e^{-\alpha t}\phi(X_t)\hat{L}_t^z\nu(dz)=-\int^t_0e^{-\alpha s}\{\mathcal{L}-\alpha\}V(X_s)ds\quad\textrm{a.s.}
\end{equation*}
Using the same arguments as in the proof of \cref{tpasting} we can show that both statements of this theorem hold. The details are left to the reader.
\end{proof}

\section*{Acknowledgments}
We are grateful to two anonymous referees and Prof. Goran Peskir for useful comments and suggestions.

\bibliographystyle{siamplain}

\appendix
\section{\color{white}}\label{appx}

\begin{lemma}\label{lem:directed}
For each $0\leq t\leq T$, the family of random variables $\{\mathbb{E}[G_\tau\lvert\mathcal{F}_t]:\tau\in\mathcal{T}_{t,T}\}$ is directed upwards, i.e. for any $\sigma_1$, $\sigma_2\in\mathcal{T}_{t,T}$, there exists $\sigma_3\in\mathcal{T}_{t,T}$, such that
\begin{equation*}
\mathbb{E}[G_{\sigma_1}\lvert\mathcal{F}_t]\vee \mathbb{E}[G_{\sigma_1}\lvert\mathcal{F}_t]\leq \mathbb{E}[G_{\sigma_3}\lvert\mathcal{F}_t],\textrm{ a.s.}
\end{equation*}
\end{lemma}
\begin{proof}
Fix $t\in[0,T]$. Suppose $\sigma_1$, $\sigma_2\in\mathcal{T}_{t,T}$ and define $A:=\{\mathbb{E}[G_{\sigma_1}\lvert\mathcal{F}_t]\geq\mathbb{E}[G_{\sigma_2}\lvert\mathcal{F}_t]\}$. Let $\sigma_3:=\sigma_1\mathbbm{1}_{A}+\sigma_2\mathbbm{1}_{A^c}$. Note that $\sigma_3\in\mathcal{T}_{t,T}$. Using $\mathcal{F}_t$-measurability of $A$, we have
\begin{align*}
\mathbb{E}[G_{\sigma_3}\lvert\mathcal{F}_t]&=\mathbbm{1}_{A}\mathbb{E}[G_{\sigma_1}\lvert\mathcal{F}_t]+\mathbbm{1}_{A^c}\mathbb{E}[G_{\sigma_2}\lvert\mathcal{F}_t]\nonumber\\
&=\mathbb{E}[G_{\sigma_1}\lvert\mathcal{F}_t]\vee\mathbb{E}[G_{\sigma_2}\lvert\mathcal{F}_t]\text{ a.s.},
\end{align*}
which proves the claim.
\end{proof}
\begin{lemma}\label{lreg1}
Let $G\in\bar{\mathbb{G}}$ and $S$ be its Snell envelope with decomposition $S=M^*-A$. For $0\leq t \leq T$ and $\epsilon>0$, define
\begin{equation}\label{eq1}
K^\epsilon_t=\inf\{s\geq t: G_s\geq S_s-\epsilon\}.
\end{equation}
Then $A_{K^\epsilon_t}=A_t$ a.s. and the processes $(A_{K^\epsilon_t})$ and $A$ are indistinguishable.
\end{lemma}
\begin{proof}
From the directed upwards property (\cref{lem:directed}) we know that $\mathbb{E}[S_t]=\sup_{\tau\in\mathcal{T}_{t,T}}\mathbb{E}[G_\tau]$. Then for a sequence $(\tau_n)_{n\in\mathbb{N}}$ of stopping times in $\mathcal{T}_{t,T}$, such that $\lim_{n\to\infty}\mathbb{E}[G_{\tau_n}]=\mathbb{E}[S_t]$, we have
\begin{align*}
\mathbb{E}[G_{\tau_n}]\leq\mathbb{E}[S_{\tau_n}]=\mathbb{E}[M^*_{\tau_n}-A_{\tau_n}]=\mathbb{E}[S_{t}]-\mathbb{E}[A_{\tau_n}-A_{t}],
\end{align*}
since $M^*$ is uniformly integrable. Hence, since $A$ is non-decreasing,
\begin{equation*}
0\leq\lim_{n\to\infty}\mathbb{E}[S_{\tau_n}-G_{\tau_n}]=-\lim_{n\to\infty}\mathbb{E}[A_{\tau_n}-A_{t}]\leq0,
\end{equation*} 
and thus we have equalities throughout. By passing to a sub-sequence we can assume that
\begin{equation}\label{eq2}
\lim_{n\to\infty}(S_{\tau_n}-G_{\tau_n})=0=\lim_{n\to\infty}(A_{\tau_n}-A_{t})\quad \text{a.s.}
\end{equation}
The first equality in \eqref{eq2} implies that $K^\epsilon_t\leq\tau_{n_0}$ a.s., for some large enough $n_0\in\mathbb{N}$, and thus $A_{K^\epsilon_t}\leq A_{\tau_n}$, for all $n_0\leq n$. Since $A$ is non-decreasing, we also have that $0\leq A_{K^\epsilon_t}-A_t\leq A_{\tau_n}-A_t$ a.s., $n_0\leq n$, and from the second equality in \eqref{eq2} we conclude that $A_{K^\epsilon_t}= A_{t}$ a.s. The indistinguishability follows from the right-continuity of $G$ and $S$.
\end{proof}
\subsection{Proofs of results in \cref{prelim}}\label{proofs2}

\begin{proof}[Proof of \cref{Fell}]
The completed filtration generated by a Feller process satisfies the \textit{usual assumptions}, in particular, it is both right-continuous and quasi-left-continuous. The latter means that for any predictable stopping time $\sigma$, $\mathcal{F}_{\sigma-}=\mathcal{F}_{\sigma}$.  Moreover, every c\`{a}dl\`{a}g Feller process is left-continuous over stopping times and satisfies the strong Markov property. On the other hand, every Feller process admits a c\`{a}dl\`{a}g modification (these are standard results and can be found, for example, in Revuz and Yor \cite{revuz2013continuous} or Rogers and Williams \cite{rogers1987markov}). All that remains is to show that the addition of the functional $F$ leaves $(X,F)$ strong Markov. This is elementary from \eqref{func}.
\end{proof}

\subsection{Proofs of results in \cref{main}}\label{proofs3}

\begin{proof}[Proof of \cref{lreg2}]
Let $(\tau_n)_{n\in\mathbb{N}}$ be a nondecreasing sequence of stopping times with $\lim_{n\to\infty}\tau_n=\tau$, for some fixed $\tau\in\mathcal{T}_{0,T}$. Since $S$ is a supermartingale, $\mathbb{E}[S_{\tau_n}]\geq\mathbb{E}[S_\tau]$, for every $n\in\mathbb{N}$. For a fixed $\epsilon>0$, $K^\epsilon_{\tau_n}$ (defined by \cref{eq1}) is a stopping time, and by \cref{lreg1}, $A_{K^\epsilon_{\tau_n}}=A_{\tau_n}$ a.s. Therefore, since $M^*$ is uniformly integrable,
\begin{equation*}
\mathbb{E}[S_{K^\epsilon_{\tau_n}}]=\mathbb{E}[M^*_{K^\epsilon_{\tau_n}}-A_{K^\epsilon_{\tau_n}}]=\mathbb{E}[M^*_{\tau_n}-A_{{\tau_n}}]=\mathbb{E}[S_{\tau_n}].
\end{equation*}
Thus, by the definition of $K^\epsilon_{\tau_n}$,
\begin{equation*}
\mathbb{E}[G_{K^\epsilon_{\tau_n}}]\geq\mathbb{E}[S_{K^\epsilon_{\tau_n}}]-\epsilon=\mathbb{E}[S_{\tau_n}]-\epsilon.
\end{equation*}
Let $\hat{\tau}:=\lim_{n\to\infty}K^\epsilon_{\tau_n}$. Note that the sequence $(K^\epsilon_{\tau_n})_{n\in\mathbb{N}}$ is non-decreasing and dominated by $K^\epsilon_{\tau}$. Hence $\tau\leq\hat{\tau}\leq K^\epsilon_{\tau}$. Finally, using the regularity of $G$ we obtain
\begin{equation*}
\mathbb{E}[S_\tau]\geq\mathbb{E}[S_{\hat{\tau}}]\geq\mathbb{E}[G_{\hat{\tau}}]=\lim_{n\to\infty}\mathbb{E}[G_{K^\epsilon_{\tau_n}}]\geq\lim_{n\to\infty}\mathbb{E}[S_{\tau_n}]-\epsilon.
\end{equation*}
Since $\epsilon$ is arbitrary, the result follows.
\end{proof}

\begin{proof}[Proof of \cref{lhloc}]
For $n\geq1$, define
\begin{equation*}
\tau_n:=\inf\{t\geq0:\int^t_0\lvert dK_s\lvert\geq n\}. 
\end{equation*}
Clearly $\tau_n\uparrow\infty$ as $n\to\infty$. Then for each $n\geq1$
\begin{align*}
\mathbb{E}[\int^{t\wedge\tau_n}_0\lvert dK_s\lvert]&\leq\mathbb{E}[\int^{\tau_n}_0\lvert dK_s\lvert]\nonumber\\
&=\mathbb{E}[\int^{\tau_n-}_0\lvert dK_s\lvert]+\lvert\Delta K_{\tau_n}\lvert]\nonumber\\
&\leq n+c.
\end{align*}
Therefore, since $X\in\mathbb{G}$,
\begin{equation*}
\lvert\lvert L^{\tau_n}\lvert\lvert_{\mathcal{S}^1}\leq\lvert\lvert X^{\tau_n}\lvert\lvert_{\mathcal{S}^1}+\mathbb{E}[\int^{\tau_n}_0\lvert dK_s\lvert]<\infty,
\end{equation*}
and thus, $\lvert\lvert X^{\tau_n}\lvert\lvert_{\mathcal{H}^1}<\infty$, for all $n\geq1$.
\end{proof}

\end{document}